\documentclass[12pt,a4paper]{article}
\usepackage{amssymb, mathrsfs, eucal, amsmath, amsthm}

\newcommand{\Hc}{{\mathcal H}}
\newcommand{\Kc}{{\mathcal K}}

\newcommand{\Bc}{{\mathcal B}}\newcommand{\Rc}{{\mathcal R}}\newcommand{\Zc}{{\mathcal Z}} 
\newcommand{\Lc}{{\mathcal L}}
\newcommand{\Nc}{{\mathcal N}}

\newcommand{\Mc}{{\mathcal M}}

\providecommand{\norm[1]}{{\left\| #1 \right\|}}

\begin{document}
\newtheorem{defn}{\bf Definition}[section]
\newtheorem{teo}{\bf Theorem}[section]
\newtheorem{lem}[teo]{\bf Lemma}
\newtheorem{prop}[teo]{\bf Proposition}
\newtheorem{cor}[teo]{\bf Corollary}
\theoremstyle{remark}
\newtheorem{ex}{\bf Example}[section]
\newtheorem{rem}{\it Remark}[section]

\title{A note on partial isometries on pseudo-Hilbert spaces}

\author{P\u{a}storel Ga\c{s}par \and Loredana Ciurdariu}
\date{ }
\maketitle
\begin{abstract}
The aim of this paper is to show that two accessible subspaces in the Loynes $\mathcal{Z}$~-~space $\Hc$ are the initial and final space of a partial gramian isometry, respectively if the norm of the difference of the associated gramian selfadjoint projections is strictly less than $1$.
\end{abstract}

\section{Introduction}

Generalizing the concept of pre-Hilbert or Hilbert space, R.M. Loynes introduced in \cite{6} the $VE$~-~spaces or $VH$~-~spaces respectively. A $VH$~-~space is characterized in \cite{7} by the fact that the inner product takes values in a suitable ordered topological vector (admissible) space ${\mathcal Z}$, thus being also called \emph{Loynes ${\mathcal Z}$~-~spaces}. Many authors used these spaces in the study of abstract stochastic processes (see \cite{8}, \cite{1}, \cite{11}, \cite{12}). In \cite{11} these spaces are referred to as \emph{pseudo-Hilbert spaces}. Spectral theory for some classes of operators on such spaces was developed initially by Loynes himself (\cite{6}, \cite{7}) and later by the authors in \cite{3}, respectively \cite{2} and by A. Gheondea and B.~E. Ugurcan in \cite{Gheondea}. \\
In what follows $\Hc$, $\Kc$ will denote two pseudo-Hilbert spaces over the same admissible space $\Zc$ and $\Lc(\Hc, \Kc)$ the space of all linear operators from $\Hc$ to $\Kc$.

Recall that an operator $T \in \Lc(\Hc, \Kc)$ is \emph{bounded}, if there exists a constant $M>0$ such that
\begin{equation}\label{eq:1}
[Th, Th]_\Kc\leq M^2[h, h]_\Hc, \qquad\quad h\in \Hc,
\end{equation}
where $[\cdot, \cdot]_{\Kc}$ is the inner product (also referred to as \emph{gramian}) of the Loynes ${\mathcal Z}$~-~space $\Kc$, while ``$\leq$'' means the order in $\Zc$. We shall denote that by $T\in\Bc(\Hc, \Kc)$. As usually, for $\Hc = \Kc$, we use the notations $\Lc(\Hc)$ and $\Bc(\Hc)$ respectively. Moreover $\Bc(\Hc, \Kc)$ is a Banach space (algebra if $\Hc= \Kc$) with the norm defined by
\[ \|T\| = \|T\|_{\Bc(\Hc, \Kc)} = \inf \{M \ :\ \eqref{eq:1}\  \text{holds}\ \}. \]
The operators $T\in \Bc(\Hc, \Kc)$ for which $\|T\| \le 1$ will be called \emph{gramian contractions}. \\
The adjoint $T^*$ of an operator $T\in\Lc(\Hc, \Kc)$ and the gramian orthogonal complement $\Mc^\perp$ of a subspace $\Mc$ of $\Hc$ will be defined (if they exist) analogously as in the Hilbert space case, but with respect to the inner products of $\Hc$ and $\Kc$. \\
By $\Lc^\ast(\Hc, \Kc), \Bc^\ast(\Hc, \Kc)$ will be denoted the set of all adjointable elements of $\Lc(\Hc, \Kc)$ and $\Bc(\Hc, \Kc)$, respectively, whereas $P_\Mc$ denotes the gramian selfadjoint projection associated to the complementable (accessible) subspace $\Mc$ of $\Hc$.\\
We also remark that $\Lc^\ast(\Hc, \Kc)\cap\Bc(\Hc, \Kc)=\Bc^\ast(\Hc, \Kc)$ and $\Bc^*(\Hc)$ is a $C^*$~-~algebra. \\
$T\in\Lc(\Hc, \Kc)$ is called a \emph{gramian isometry} (\emph{gramian co-isometry})
if $T \in \Bc^*(\Hc, \Kc)$ and $T^*T = I_\Hc$ ($T T^* = I_\Kc$, respectively) and $T$ is \emph{gramian unitary} if it is simultaneously a gramian isometry and a gramian co-isometry. \\
If a gramian contraction $T$ is adjointable, then $T^\ast$ is a gramian contraction too (see \cite{6}, \cite{7}).
Familiar examples of adjointable contractions are self-adjoint gramian projections and gramian partial isometries, the latter containing two remarkable subclasses: that of gramian isometries and of gramian co-isometries. The latter classes were studied by the first author in \cite{3}, where also a geometric proof of the existence of the gramian co-isometric extension of a gramian adjointable contraction is given. \\
In what follows some definitions and results from \cite{3} are needed.
\begin{defn}
Let $\Hc$ and $\Kc$ be two Loynes $\mathcal{Z}$-spaces. A linear operator $T\in\Lc(\Hc, \Kc)$ is a partial gramian isometry, if its kernel $\Nc(T)$ and its range $\Rc(T)$ are accessible (i.e. they have gramian orthogonal complements) in $\Hc$ and $\Kc$, respectively and from $\Nc(T)^\perp$ to $\Rc(T)$ it preserves the gramian (is gramian unitary). The spaces $\Mc:=\Nc(T)^\perp$ and $\Rc(T)$ are called the initial and the final space of $T$, respectively. The set of all partial gramian isometries  from $\Hc$ to $\Kc$ will be denoted by $\mathcal{PI}(\Hc, \Kc)$.
\end{defn}
It can be easily seen, that $\mathcal{PI}(\Hc, \Kc)\subset\Bc^\ast(\Hc, \Kc)$. Observe that if $\Nc(T)=0$, then $T$ is simply a gramian isometry.

\begin{prop}
If $T\in\mathcal{PI}(\Hc, \Kc)$, then $T^\ast T$ and $TT^\ast$ are gramian self-adjoint projections for which the following hold:
\begin{itemize}
\item[(i)] $T^\ast T=P_{\Mc(T)}$;
\item[(ii)] $P_{\Rc(T)}=TT^\ast$;
\item[(iii)] If $\Hc=\Kc$ then $T$ is a partial isometry in $\Bc^\ast(\Hc)$ as a $C^\ast$-algebra.
\end{itemize}

\end{prop}


\begin{prop}
\label{prop2} For $T\in\Lc(\Hc, \Kc)$, the following are equivalent:
\begin{itemize}
\item[(i)] $T\in \mathcal{PI}(\Hc, \Kc)$;
\item[(ii)] $T\in\Bc^\ast(\Hc, \Kc)$ and $T^\ast T$ is a gramian self-adjoint projection on $\Hc$;
\item[(iii)] $T\in \Bc^\ast(\Hc, \Kc)$ and $TT^\ast$ is a gramian self-adjoint projection on $\Kc$;
\item[(iv)] $T\in\Lc^\ast(\Hc, \Kc)$ and $T^\ast\in\mathcal{PI}(\Kc, \Hc)$.
\end{itemize}
\end{prop}

It is obvious that any gramian isometry or gramian co-isometry is a partial gramian isometry.

\section{The result}

Focusing on the case $\Hc=\Kc$ and taking $T\in \mathcal{PI}(\Hc)$, then the operators $T^\ast T$ and $TT^\ast$ will be two gramian self-adjoint projections in $\Bc^\ast(\Hc)$. It is thus interesting, as in the case of Hilbert space (see \cite[pp. 266,267]{10}), to find a sufficient condition on two gramian self-adjoint projections $P$ and $Q$ in order to have their ranges as initial and final space of a certain partial gramian isometry. Indeed the following assertion holds.

\begin{teo}\label{thm:carPQ}
If $P$ and $Q$ are gramian self-adjoint projections and
\begin{equation}\label{eq:ineg}
\norm{P-Q}<1,
\end{equation}
then there exists $T\in\mathcal{PI}(\Hc)$ such that $P=T^\ast T$ and $Q=TT^\ast$.
\end{teo}
\begin{proof}
Denote $A=I+P(Q-P)P$. Since $\norm{I-A}=\norm{P(Q-P)P}\leq\norm{P-Q}<1$, by using that $\Bc^\ast(\Hc)$ is a Banach algebra, it results that $A$ is invertible with a bounded inverse. On the other hand the operator $A$ is positive. Indeed
\begin{equation*}
\begin{split}
[Ah, h]&=[h, h]+[P(Q-P)Ph, h]=[h, h]+[QPh, h]-[P^3h, h]\\
&=[(I-P)h, h]+[QPh, Ph]\geq0,
\end{split}
\end{equation*}
where we used the fact that $I-P$ and $Q$ are gramian self-adjoint projections. In this situation, there exists the square root of $A$, which is also invertible. The operator $T:=QA^{-1/2}P$ satisfies the requirements of the statement. Indeed we have $T^\ast=PA^{-1/2}Q$ and further on $PT^\ast=T^\ast=A^{-1/2}PQ$. Since $PA=AP$, we infer that $PA^{1/2}=A^{1/2}P$ which implies $A^{-1/2}P=PA^{-1/2}$. Further we get $TP=T=QPA^{-1/2}$. Taking into account that $PA=PQP$ we infer
\begin{equation*}
T^\ast T=A^{-1/2}PQQPA^{-1/2}=A^{-1/2}PQPA^{-1/2}=A^{-1/2}PAA^{-1/2}=P.
\end{equation*}
Using Proposition \ref{prop2} we infer that $T$ is a partial gramian isometry and $TT^\ast=P_{\Rc(T)}$. But, the calculation of $TT^{\ast}$ leads us to the equalities
\begin{equation*}
TT^\ast=QA^{-1/2}PPA^{-1/2}Q=QA^{-1}PQ,
\end{equation*}
which imply $\Rc(T)=\Rc(TT^\ast)\subset\Rc(Q)$, i.e. $TT^\ast\leq Q$. Now, let us show that $I-TT^\ast\leq I-Q$. Let $h\in(I-TT^\ast)\Hc$. Then the next implications hold
\begin{equation*}
\begin{split}
h\in(I-TT^\ast)\Hc & \Rightarrow h =(I-TT^\ast)h\Rightarrow TT^\ast h=0\\
&\Rightarrow R^\ast h\in\Nc(T)\cap\Rc(T^\ast)=\{0\}\Rightarrow T^\ast h=0\\
&\Rightarrow PA^{-1/2}Qh=0\Rightarrow PQh=0\Rightarrow (Q-P)Qh=Qh \\
& \Rightarrow Qh=0\Rightarrow\ (I-Q)h=h \\
& \Rightarrow h\in \Rc(I-Q).
\end{split}
\end{equation*}
This shows that $\Rc(I-TT^\ast)\subset\Rc(I-Q)$, which indicates that $I-TT^\ast\leq I-Q$. Hence the equality $TT^\ast=Q$ holds.
\end{proof}
\begin{rem}
Our theorem can be applied in perturbation theory to treat the variation of the spectral measure of gramian selfadjoint operators on pseudo-Hilbert spaces in a limit taking process. For the Hilbert space case see \cite{10} no. 135.
\end{rem}
\begin{rem}
Our theorem states that \eqref{eq:ineg} is a sufficient condition on the two gramian selfadjoint projections $P$ and $Q$ in order to determine the initial and final space of a partial isometry. This condition isn't however necessary, as the following example shows. For $V$ a gramian (non-unitary) isometry on $\Hc$ we have that $V^*V - VV^*$, being a gramian selfadjoint projection, has norm equal to $1$.

It would therefore be interesting to find a weaker condition that would still be sufficient.
\end{rem}
\begin{rem}
Our definition of the partial isometry on the pseudo-Hilbert space $\Hc$ as well as the statement of our result being given in the $C^*$~-~algebra $\Bc^*(\Hc)$ let us observe that following \cite{mbekhta} or \cite{MosicDjordjevic} it is possible to define and characterize the notion of a partial isometry in the Banach algebra $\Bc(\Hc) ( \supset \Bc^*(\Hc) )$. It is then naturally to ask if there exist such partial isometries in $\Bc(\Hc)$ which are not in $\Bc^*(\Hc)$ and if this would be the case, would an analogue of Theorem \ref{thm:carPQ} hold in $\Bc(\Hc)$ ?
\end{rem}

P\u{a}storel Ga\c{s}par, ``Aurel Vlaicu'' University, Arad, e-mail:{\tt pastogaspar@yahoo.com} \\
Loredana Ciurdariu, ``Politehnica'' University, Timi\c soara


\begin{thebibliography}{99}
\bibitem{1} S.A. Chobanyan, A. Weron, {\it Banach-space-valued stationary processes and thei linear prediction}, Dissertationes Math., 125, (1975), 1-45.

\bibitem{2} L. Ciurdariu, A. Cr\u aciunescu, {\it On Spectral Representation of Gramian Normal Operators on Pseudo-Hilbert Spaces}, Anal. Univ. de Vest Timi\c soara, Vol. XLV, (1), 2007, 131-149.

\bibitem{3} P. Ga\c spar, {\it Partial isometries on Loynes spaces}, Anal. Univ. de Vest Timi\c soara, Vol. XL, (2), 2002, 31-47.



\bibitem{Gheondea} A. Gheondea, B. E. Ugurcan, {\it On two equivalent dilation theorems in $VH$~-~spaces}, Complex Analysis and Operator Theory, vol. 6, (3), 2012, 625-650.

\bibitem{6} R.M. Loynes, {\it Linear operators in $VH$-spaces}, Trans. American Math. Soc., 116, (1965), 167-180.

\bibitem{7} R.M. Loynes, {\it On generalized positive definite functions}, Proc. London Math. Soc., 3, (1965), 373-384.

\bibitem{8} R.M. Loynes, {\it On a generalization of second-order stationarity}, Proc. London Math. Soc., 3, (1965), 385-398.

\bibitem{mbekhta} M. Mbekhta, {\it Partial isometries and generalized inverses}, Acta. Sci. Math. (Szeged) {\bf 70} (2004), 767-781.

\bibitem{MosicDjordjevic} D. Mosi\v c, D. S. Djordjevi\v c, {\it Partial isometries and EP elements in Banach algebras}, preprint, \verb+http://operator.pmf.ni.ac.rs/+
    \verb+licne_prezentacije/DDjordjevic/publications/+
    \verb+Partial-EP-Banach.pdf+

\bibitem{10} F. Riesz, B.Sz.-Nagy, {\it Le\c cons d'analyse fonctionnelle}, 4-\`eme Edition, Gauthier-Villars Paris and Akad\'emiai Kiad\'o Budapest, 1965.

\bibitem{11} A. Weron, S.A. Chobanyan, {\it Stochastic Processes on pseudo-Hilbert spaces(russian)}, Bull. Acad. Pol., Ser. Math. Phys., tom XX1, 9, (1973), 847-854.

\bibitem{12} A. Weron, {\it Prediction theory in Banach spaces}, Proc. of Winter School on Probability, Karpacz, Springer Verlag, London, 1975, 207-228.

\end{thebibliography}
\end{document}